\documentclass[a4paper,12pt]{amsart}

\makeindex

\usepackage{amsfonts,graphics,amsmath,amsthm,amsfonts,amscd,amssymb,amsmath,latexsym,euscript, enumerate}
\usepackage{epsfig,url}
\usepackage{flafter}
\usepackage[all,cmtip,line]{xy}
\usepackage{array}
\usepackage[english]{babel}
\usepackage{overpic}
\usepackage{subfig}
\usepackage{multirow}

\usepackage{wrapfig}
\usepackage[shortlabels]{enumitem}

\usepackage[usenames,dvipsnames,svgnames,table]{xcolor}
\usepackage{graphicx}
\usepackage[bookmarks=true,bookmarksnumbered=true,
 pdftitle={},
 pdfsubject={},
 pdfauthor={},
 pdfkeywords={TeX, dvipdfmx, hyperref, color,},
 colorlinks=true, linkcolor=RoyalBlue, citecolor=Emerald]
 {hyperref}

\newtheorem{theorem}{Theorem}[section]
\newtheorem{lemma}[theorem]{Lemma}
\newtheorem{proposition}[theorem]{Proposition}

\theoremstyle{definition} 
\newtheorem{definition}[theorem]{Definition}
\newtheorem{definition-lemma}[theorem]{Definition-Lemma}

\newtheorem{remark}[theorem]{Remark}

\numberwithin{equation}{section}

\newcommand{\R}{\mathbb{R}}

\newcommand{\Q}{\mathbb{Q}}
\newcommand{\mc}{\mathcal}

\DeclareMathOperator{\Nef}{Nef}
\DeclareMathOperator{\Div}{Div}

\DeclareMathOperator{\NE}{NE}
\DeclareMathOperator{\Mov}{Mov}

\DeclareMathOperator{\Fix}{Fix}

\DeclareMathOperator{\Exc}{Exc}
\DeclareMathOperator{\Supp}{Supp}
\DeclareMathOperator{\codim}{codim}
\DeclareMathOperator{\lct}{lct}

\DeclareMathOperator{\pt}{pt}

\def\mult{\operatorname{mult}}

\newcommand{\pp}{P_{\sigma}}
\newcommand{\np}{N_{\sigma}}

\usepackage{mathtools}

\DeclarePairedDelimiterX{\norm}[1]{\lVert}{\rVert}{#1}

\title[Anticanonical minimal models and Zariski decomposition]
{Anticanonical minimal models and Zariski decomposition}

\begin{document}

\author{Sungwook Jang}
\address[Sungwook Jang]{Center for Complex Geometry, Institute for Basic Science, 34126 Daejeon, Republic of Korea}
\email{swjang@ibs.re.kr}

\date{\today}

\begin{abstract}
Birkar and Hu showed that if a pair $(X,\Delta)$ is lc and $K_{X}+\Delta$ admits a birational Zariski decomposition, then $(X,\Delta)$ has a minimal model. Analogously, we prove that if a pair $(X,\Delta)$ is pklt and $-(K_{X}+\Delta)$ admits a birational Zariski decomposition, then $(X,\Delta)$ has an anticanonical minimal model.
\end{abstract}

\maketitle

\section{Introduction}
Ambro (\cite{Amb}) and Fujino (\cite{Fuj}) proved the cone theorem and contraction theorem for lc pairs. As a consequence, we can initiate a minimal model program (abbreviated as MMP) for lc pairs. Birkar (\cite{Bir12}) and Hacon-Xu (\cite{HX}) proved the existence of flips for lc pairs. It allows us to keep running the MMP (cf. Remark \ref{remark:D-MMP}). However, we do not know whether the MMP terminates in finite steps. Nevertheless, there are results which support the existence of minimal models. In \cite{BCHM}, the authors showed the existence of minimal models for klt pairs of general type. After \cite{BCHM}, Birkar proved the existence of minimal models under various conditions (cf. \cite{Bir10, Bir11, Bir12, BH}).

Let $D$ be an $\R$-Cartier divisor on a normal projective variety $X$. We say that $D$ admits a birational Zariski decomposition if there exists a projective birational morphism $f:Y\to X$ such that the positive part $\pp(f^{\ast}D)$ of the divisorial Zariski decomposition of $f^{\ast}D$ is nef. It is easy to see that if an lc pair $(X,\Delta)$ has a minimal model, then $K_{X}+\Delta$ admits a birational Zariski decomposition (cf. Section \ref{remark:minimal model-ZD}). In \cite{BH}, Birkar and Hu conversely showed that if $K_{X}+\Delta$ admits a birational Zariski decomposition, then $(X,\Delta)$ has a minimal model.

\begin{theorem}[\cite{BH}]
Let $(X,\Delta)$ be an lc pair with $K_{X}+\Delta$ pseudoeffective. If $K_{X}+\Delta$ admits a birational Zariski decomposition, then $(X,\Delta)$ has a minimal model.
\end{theorem}

In this paper, we are interested in whether the similar result holds when we change the sign of $K_{X}+\Delta$, that is, we replace $K_{X}+\Delta$ with $-(K_{X}+\Delta)$. It is not surprising that this small modification causes the whole argument in the proof to the break down. First of all, it is not at all trivial whether the MMP on $-(K_{X}+\Delta)$ works or not. Furthermore, even  if we can run such MMP, the singularities of the resulting model can be arbitrary bad. Thus, we need to control singularities carefully.

Let $(X,\Delta)$ be a pair with $-(K_{X}+\Delta)$ pseudoeffective and let $E$ be a prime divisor over $X$. We define the \textit{potential log discrepancy} $\bar{a}(E;X,\Delta)$ of $(X,\Delta)$ at $E$ by $\bar{a}(E;X,\Delta):=a(E;X,\Delta)-\sigma_{E}(-(K_{X}+\Delta))$, where $a(E;X,\Delta)$ is the usual log discrepancy and $\sigma_{E}(-(K_{X}+\Delta))$ is the asymptotic valuation. We say that $(X,\Delta)$ is \textit{pklt} if $\inf_{E}\bar{a}(E;X,\Delta)>0$, where the infimum is taken over all prime divisors $E$ over $X$. These notions were first introduced by Choi and Park in \cite{CP}.

\begin{definition} \label{definition:-K-minimal model}
Let $(X,\Delta)$ be a pair with $-(K_{X}+\Delta)$ pseudoeffective. We say that a pair $(X',\Delta')$ is a \textit{$-(K_{X}+\Delta)$-minimal model} if
\begin{enumerate}[label=$\bullet$]
\item there exists a birational contraction $\phi:X\dashrightarrow X'$, i.e., $\phi$ is a birational map and $\phi^{-1}$ does not contract any divisor,
\item $\Delta'=\phi_{\ast}\Delta$ is the birational transform of $\Delta$ on $X'$,
\item $(X',\Delta')$ is $\Q$-factorial lc,
\item $-(K_{X'}+\Delta')$ is nef, and
\item $\bar{a}(D;X,\Delta)=\bar{a}(D;X',\Delta')$ for any $\phi$-exceptional prime divisor $D$.
\end{enumerate}
\end{definition}

Now, we are ready to state our main result.

\begin{theorem} \label{theorem:main theorem}
Let $(X,\Delta)$ be a pair with $-(K_{X}+\Delta)$ pseudoeffective . Assume that $(X,\Delta)$ is pklt and $-(K_{X}+\Delta)$ admits a birational Zariski decomposition. Then there exists a $-(K_{X}+\Delta)$-minimal model.
\end{theorem}

\section*{Acknowledgement}

We thank Professor Sung Rak Choi for helpful suggestions and careful comments on the earlier version of this paper. We also thank Professor Chuyu Zhou for kind suggestion and corrections.

\section{Preliminaries}
\subsection{Numerical equivalence}
Let $X$ be a complete variety. Denote by $\Div_{\R}(X)$ the real vector space generated by Cartier divisors on $X$. An element of $\Div_{\R}(X)$ is called \textit{$\R$-Cartier divisor} and we say that two $\R$-Cartier divisors $D_{1},D_{2}\in \Div_{\R}(X)$ are \textit{numerically equivalent}, denoted by $D_{1}\equiv D_{2}$, if $D_{1}\cdot C=D_{2}\cdot C$ for every integral curve $C$ on $X$. Let $Z_{1}(X)$ be the real vector space generated by irreducible curves on $X$. An element of $Z_{1}(X)$ is called \textit{1-cycle} and we say that two 1-cycles $\gamma_{1},\gamma_{2}\in Z_{1}(X)$ are \textit{numerically equivalent}, denoted by $\gamma_{1}\equiv \gamma_{2}$, if $D\cdot \gamma_{1}=D\cdot \gamma_{2}$ for every Cartier divisor $D$ on $X$. Define two quotient spaces $N^{1}(X), N_{1}(X)$ by $N^{1}(X):=\Div_{\R}(X)/\equiv$ and $N_{1}(X):=Z_{1}(X)/\equiv$, respectively. Then intersection number gives a perfect pairing
$$ N^{1}(X)\times N_{1}(X)\to \R,~~(\delta,\gamma)\mapsto \delta\cdot \gamma. $$

An $\R$-Cartier divisor $D$ on $X$ is \textit{nef} if $D\cdot C\ge 0$ for any irreducible curve $C$ on $X$. A \textit{nef cone} $\Nef(X)$ is the cone in $N^{1}(X)$ generated by nef $\R$-Cartier divisors. A \textit{Mori cone} $\overline{\NE}(X)$ is the cone in $N_{1}(X)$ generated by irreducible curves on $X$. By definition, one can see that $\Nef(X)$ is the dual cone of $\overline{\NE}(X)$.

\subsection{Asymptotic valuation}
In this section, $X$ denotes a smooth projective variety unless otherwise specified. Let $E$ be a prime divisor on $X$. For a big $\R$-Cartier divisor $D$ on $X$, we define the \textit{asymptotic valuation} $\sigma_{E}(D)$ of $D$ along $E$ by
$$ \sigma_{E}(D):=\inf\{\mult_E\Delta ~|~0\le \Delta\sim_{\Q}D\}.$$
By definition, one can easily see that the following inequality
\begin{equation} \label{equation:ch2-1}
\sigma_{E}(D_{1}+D_{2})\le \sigma_{E}(D_{1})+\sigma_{E}(D_{2})
\end{equation}
holds for any big $\R$-divisors $D_{1}, D_{2}$ and for any prime divisor $E$ on $X$. Let $A$ be an ample $\R$-divisor on $X$. Then $\sigma_{E}(A)=0$. Indeed, by inequality (\ref{equation:ch2-1}), we may assume that $A$ is a very ample divisor on $X$ and obviously $\sigma_{E}(A)=0$ in this case.

\begin{lemma}[\protect{\cite[Lemma III.1.4]{Nak}}] \label{lemma:asymptotic valuation big divisor}
Let $E$ be a prime divisor and $D$ a big $\R$-Cartier divisor on $X$. Then
$$ \sigma_{E}(D)=\lim_{\varepsilon\searrow 0}\sigma_{E}(D+\varepsilon A), $$
where $A$ is an ample $\R$-divisor on $X$. In particular, the limit does not depend on the choice of an ample $\R$-divisor.
\end{lemma}

\begin{proof}
For any $\varepsilon>0$, we have
$$ \sigma_{E}(D+\varepsilon A)\le \sigma_{E}(D)+\sigma_{E}(\varepsilon A)=\sigma_{E}(D). $$
On the other hand, since $D$ is big, there exist a positive rational number $\delta>0$ and an effective $\R$-divisor $\Delta$ such that $D\sim_{\Q}\delta A+\Delta$ (cf. \cite[Lemma 7.19]{Fuj2}). Hence we have
\begin{align*}
(1+\varepsilon)\sigma_{E}(D)&=\sigma_{E}((1+\varepsilon)D)\\
&=\sigma_{E}(D+\varepsilon\delta A+\varepsilon\Delta)\\
&\le \sigma_{E}(D+\varepsilon\delta A)+\varepsilon\mult_{E}\Delta.
\end{align*}
Thus, we obtain that
\begin{equation*}
\sigma_{E}(D)\le\liminf \sigma_{E}(D+\varepsilon A)\le \limsup \sigma_{E}(D+\varepsilon A)\le \sigma_{E}(D). \qedhere
\end{equation*}
\end{proof}

Let $D$ be a pseudoeffective $\R$-divisor and $A'$ an ample $\R$-divisor on $X$. Then there exist a positive rational number $\delta$ and an effective $\R$-divisor $\Delta$ such that $A'\sim_{\Q}\delta A+\Delta$. Using the similar argument in the proof of Lemma \ref{lemma:asymptotic valuation big divisor}, we have
$$ \lim_{\varepsilon\searrow 0}\sigma_E(D+\varepsilon A)=\lim_{\varepsilon\searrow 0}\sigma_E(D+\varepsilon A'). $$
Thus the following is well-defined.

\begin{definition}
Let $E$ be a prime divisor on $X$ and $D$ a pseudoeffective $\R$-Cartier divisor on $X$. Then we define
$$ \sigma_{E}(D):=\lim_{\varepsilon\searrow 0}\sigma_E(D+\varepsilon A), $$
where $A$ is an ample divisor on $X$.
\end{definition}

\begin{remark} \label{remark:asymptotic valuation-num}
We remark that $\sigma_{E}$ is invariant under numerical equivalence. For a big $\R$-Cartier divisor $D$ on $X$, define $\sigma_{E}(D)_{\text{num}}$ as
$$ \sigma_{E}(D)_{\text{num}}:=\inf\{\mult_{E}\Delta ~|~0\le \Delta\equiv D\}. $$
Obviously, $0\le\sigma_{E}(D)_{\text{num}}\le\sigma_{E}(D)$. Let $A$ be an ample $\R$-divisor and $\Delta$ an effective divisor such that $\Delta\equiv D$. Then $D+\varepsilon A-\Delta$ is an ample $\R$-divisor for any $\varepsilon>0$. Therefore,
$$ \sigma_{E}(D+\varepsilon A)\le \mult_{E}\Delta+\sigma_{E}(D+\varepsilon A-\Delta)=\mult_{E}\Delta $$
for any $\varepsilon>0$. This shows that $\sigma_{E}(D)_{\text{num}}=\sigma_{E}(D)$. In particular, if $D'$ is an $\R$-divisor such that $D\equiv D'$, then $\sigma_{E}(D)=\sigma_{E}(D')$.
\end{remark}

Let $D$ be a pseudoeffective $\R$-Cartier divisor on $X$. The \textit{negative part} $\np(D)$ of $D$ is defined as
$$ \np(D)=\sum_{E}\sigma_{E}(D)E $$
and the \textit{positive part} $\pp(D)$ of $D$ is defined as $\pp(D)=D-\np(D)$. By \cite[Corollary III.1.11]{Nak}, the above summation is finite. We call $D=\pp(D)+\np(D)$ the \textit{divisorial Zariski decomposition}. If $\pp(D)$ is nef, then we call $D=\pp(D)+\np(D)$ the \textit{Zariski decomposition}. 

Now, we extend the notion of asymptotic valuation to a singular variety. Let $X$ be a normal projective variety and $D$ a pseudoeffective $\R$-Cartier divisor on $X$. Suppose that $E$ is a prime divisor over $X$. Then there exists a birational morphism $f:Y\to X$ such that $Y$ is a smooth projective variety and $E$ is a prime divisor on $Y$. We define as $\sigma_{E}(D):=\sigma_{E}(f^{\ast}D)=\mult_{E}\np(f^{\ast}D)$. It does not depend on the choice of $f$. Moreover, following \cite{BH}, we can define the \textit{positive part} $\pp(D)$ of $D$ as $\pp(D)=f_{\ast}\pp(f^{\ast}D)$. Similarly, the \textit{negative part} $\np(D)$ of $D$ is defined as $\np(D)=f_{\ast}\np(f^{\ast}D)$. We also call $D=\pp(D)+\np(D)$ the \textit{divisorial Zariski decomposition}. Note that these definitions do not depend on the choice of $f$. Furthermore, if $\pp(D)$ is nef, then the decomposition $D=\pp(D)+\np(D)$ is called \textit{Zariski decomposition}. If in addition $\pp(D)$ is semiample, we say that the Zariski decomposition is \textit{good}. We say that $D$ admits a \textit{birational Zariski decomposition} if there exists a birational morphism $f:Y\to X$ such that $Y$ is a smooth projective variety and $\pp(f^{\ast}D)$ is nef.

\subsection{Movable divisors}
Let $X$ be a normal projective variety and $D$ a pseudoeffective $\R$-Cartier divisor on $X$. We say that $D$ is \textit{movable} if $\np(D)=0$. We also say that a class in $N^{1}(X)$ is \textit{movable} if it is represented by movable divisor. Denote by $\overline{\Mov}(X)$ the closure of the cone generated by movable classes in $N^{1}(X)$. Then by \cite[Lemma 1.7]{Nak}, $\overline{\Mov}(X)$ coincides with the closure of the convex hull of fixed part free divisors $D$, i.e. $\Fix|D|=0$.

\subsection{Potential log discrepancy}
Let $X$ be a normal projective variety. An effective $\Q$-divisor $\Delta$ on $X$ is said to be a boundary divisor if its coefficients are less than or equal to 1. A \textit{pair} $(X,\Delta)$ consists of a normal projective variety $X$ and a boundary $\Q$-divisor $\Delta$ on $X$ such that $K_{X}+\Delta$ is $\Q$-Cartier. If we allow $\Delta$ to have negative coefficients, then $(X,\Delta)$ is called \textit{sub-pair}.

Let $(X,\Delta)$ be a pair and $E$ a prime divisor over $X$. Then there is a projective birational morphism $f:Y\to X$ such that $Y$ is smooth and $E$ is a divisor on $Y$. The \textit{log discrepancy} $a(E;X,\Delta)$ of $(X,\Delta)$ at $E$ is defined as $a(E;X,\Delta)=1+\mult_{E}(K_{Y}-f^{\ast}(K_{X}+\Delta))$. We say that $(X,\Delta)$ is klt if $\inf_{E}a(E;X,\Delta)>0$, where the infimum is taken over all prime divisors $E$ over $X$. The log discrepancy is also defined for a sub-pair by the same way and we say that a sub-pair $(X,\Delta)$ is \textit{sub-klt} if it satisfies the same inequality above.

Let $(X,\Delta)$ be a pair with $-(K_{X}+\Delta)$ pseudoeffective and $E$ a prime divisor over $X$. The \textit{potential log discrepancy} $\bar{a}(E;X,\Delta)$ of $(X,\Delta)$ at $E$ is defined as
$$ \bar{a}(E;X,\Delta)=a(E;X,\Delta)-\sigma_{E}(-(K_{X}+\Delta)), $$
where $\sigma_{E}(-(K_{X}+\Delta))$ is the asymptotic valuation. We say that $(X,\Delta)$ is \textit{pklt} if $\inf_{E}\bar{a}(E;X,\Delta)>0$, where the infimum is taken over all prime divisors $E$ over $X$. The notion of potential log discrepancy is first introduced and studied in \cite{CP}. 

Let $\phi:X\dashrightarrow X'$ be a birational contraction of normal projective varieties, $D$ an $\R$-Cartier divisor on $X$, and let $D'=\phi_{\ast}D$ be the birational transform of $D$. We say that $\phi$ is \textit{$D$-nonpositive} if there exists a common log resolution $p:W\to X$ and $q:W\to X'$ and we can write
$$ p^{\ast}D=q^{\ast}D'+E, $$
where $E$ is an effective $q$-exceptional divisor. The potential log discrepancy of $(X,\Delta)$ is preserved along a $-(K_{X}+\Delta)$-nonpositive map. It is a one of significant properties of potential log discrepancy.

\begin{proposition}[\protect{\cite[Proposition 3.11]{CP}}]
Let $(X,\Delta)$ be a pair with $-(K_{X}+\Delta)$ pseudoeffective. Suppose that there is a $-(K_{X}+\Delta)$-nonpositive map $\phi:X\dashrightarrow X'$, and let $\Delta'$ be the birational transform of $\Delta$ on $X'$. Then we have
$$\bar{a}(E;X,\Delta)=\bar{a}(E;X',\Delta')$$
for any prime divisor $E$ over $X$.
\end{proposition}

\section{Negative parts of the divisorial Zariski decompositions}
In this section, we will investigate the properties of the negative part of the divisorial Zariski decomposition. To do so, we need a general negativity lemma.

\subsection{General negativity lemma}
A projective morphism $f:X\to Z$ of quasi-projective varieties is called a \textit{contraction} if $f_{\ast}\mc{O}_{X}=\mc{O}_{Z}$. Let $f:X\to Z$ be a contraction. We say that a Cartier divisor $D$ on $X$ is \textit{$f$-nef} (or \textit{nef/$Z$}) if $D\cdot C\ge 0$ for every irreducible curve such that $f(C)$ is a point in $Z$.

\begin{lemma}[cf.\protect{\cite[Lemma 3.39]{KM}}] \label{lemma:negativity lemma}
Let $f:X\to Z$ be a proper birational morphism of normal varieties. Let $B$ be a $\Q$-Cartier $\Q$-divisor on $X$ such that $-B$ is $f$-nef. Then $B$ is effective if and only if $f_{\ast}B$ is effective.
\end{lemma}

The above lemma, called the negativity lemma, is useful in many situations. There are generalizations of the negativity lemma (cf. \cite[Lemma 3.22]{Sho}, \cite[Lemma 1.7]{Prok}).

\begin{definition}
Let $f:X\to Z$ be a contraction of normal varieties, $D$ an $\R$-Cartier divisor on $X$, and $V\subseteq X$ a closed subset. We say that $V$ is \textit{vertical} over $Z$ if $f(V)$ is a proper subset of $Z$. We say that $D$ is \textit{very exceptional} over $Z$ if $D$ is vertical/$Z$ and for any prime divisor $P$ on $Z$, there exists a prime divisor $Q$ on $X$ such that $Q$ is not a component of $D$ and $f(Q)=P$, i.e., over the generic point of $P$, we have $\Supp f^{\ast}P\not\subseteq \Supp D$.
\end{definition}

If $\codim f(D)\ge 2$, then $D$ is very exceptional. On the other hand, when $f$ is birational, then the notions of exceptional and very exceptional coincide.

Let $S\to Z$ be a projective morphism of varieties and $M$ an $\R$-Cartier divisor on $S$. We say that $M$ is \textit{nef on the very general curves} of $S/Z$ if there is a countable union $\Lambda$ of proper closed subsets of $S$ such that $M\cdot C\ge 0$ for any curve $C$ on $S$ contracted over $Z$ satisfying $C\not\subseteq \Lambda$.

The following is a generalization of the negativity lemma:

\begin{lemma}[\protect{\cite[Lemma 3.22]{Sho}}] \label{lemma:general negativity lemma}
Let $f:X\to Z$ be a contraction of normal varieties. Let $D$ be an $\R$-Cartier divisor on $X$ written as $D=D^{+}-D^{-}$ with $D^{+},D^{-}\ge 0$ having no common components. Assume that $D^{-}$ is very exceptional/$Z$, and that for each component $S$ of $D^{-}$, $-D|_{S}$ is nef on the very general curves of $S/Z$. Then $D\ge 0$.
\end{lemma}

For the proof of the above lemma, we refer to \cite{Bir12}.

\subsection{Behavior of the negative part}
Following lemmas describe the behavior of the negative part of a divisorial Zarsiki decomposition. All lemmas have their origins in \cite{Nak} or \cite{BH}. For the convenience, we presents the proofs here.

\begin{lemma}[cf.\protect{\cite[Theorem III.5.16]{Nak}}] \label{lemma:np-behavior}
Let $X$ be a normal projective variety and $D$ a pseudoeffective $\R$-divisor on $X$. Let $f:Y\to X$ be a birational morphism from a smooth projective variety $Y$. Then there exists an effective $f$-exceptional divisor $\Gamma$ on $Y$ such that
$$ \np(f^{\ast}D)=f^{\ast}\np(D)+\Gamma. $$
Furthermore, if $\pp(D)$ is nef, then $\np(f^{\ast}D)=f^{\ast}\np(D)$.
\end{lemma}

\begin{proof}
It is an immediate consequence of Lemma \ref{lemma:general negativity lemma}. Note that $\np(f^{\ast}D)-f^{\ast}\np(D)$ is $f$-exceptional and
$$ f^{\ast}(D-\np(D))=\pp(f^{\ast}D)+\np(f^{\ast}D)-f^{\ast}\np(D). $$
By Lemma \ref{lemma:general negativity lemma}, $\Gamma:=\np(f^{\ast}D)-f^{\ast}\np(D)$ is effective.

Assume that $\pp(D)$ is nef. Then $f^{\ast}\pp(D)$ is also nef and $f^{\ast}\np(D)\ge \np(f^{\ast}D)$ by definition. Hence, $f^{\ast}\np(D)= \np(f^{\ast}D)$.
\end{proof}

\begin{lemma}[\protect{\cite[Lemma III.1.8]{Nak}}] \label{lemma:np-difference}
Let $X$ be a smooth projective variety and $D$ a pseudoeffective $\R$-divisor on $X$. Assume that $0\le N\le \np(D)$. Then $\np(D-N)=\np(D)-N$.
\end{lemma}

\begin{proof}
When $D$ is big, the assertion immediately follows from the definition (cf. \cite[Lemma III.1.4-(4)]{Nak}). Now, assume that $D$ is pseudoeffective. Let $A$ be an ample $\R$-divisor on $X$ and let $0<\varepsilon<1$. Then for all sufficiently small $\delta>0$,
$$ \np(D+\delta A)\ge (1-\varepsilon)\np(D). $$
Since $D+\delta A$ is big, we have
$$ \np(D+\delta A-(1-\varepsilon)N)=\np(D+\delta A)-(1-\varepsilon)N. $$
Thus,
$$ \np(D-(1-\varepsilon)N)=\np(D)-(1-\varepsilon)N. $$
Letting $\varepsilon\to 0$, by the lower semicontinuity (\cite[Lemma III.1.7]{Nak}), we have
$$ \np(D-N)\le \np(D)-N. $$
On the other hand, by the convexity (\ref{equation:ch2-1}),
$$ \np(D)\le \np(D-N)+N. $$
Hence, we have the desired result.
\end{proof}

\begin{remark} \label{remark:np-addition of pp}
Let $X$ be a smooth projective variety and $D$ a pseudoeffective $\R$-divisor on $X$. Denote by $P=\pp(D), N=\np(D)$. Then for any $x\ge 0$, we have
$$ \np(D+xP)=\np((1+x)D-xN)=(1+x)N-xN=N. $$
\end{remark}

\begin{lemma}[\protect{\cite[Lemma III.5.14]{Nak}}]\label{lemma:np-exceptional divisor}
Let $X$ be a normal projective variety and $f:Y\to X$ a resolution of $X$. Let $D$ be an $\R$-Cartier divisor on $X$ and $E$ an effective $f$-exceptional divisor. Then
$$ \np(f^{\ast}D+E)=\np(f^{\ast}D)+E. $$
\end{lemma}

\begin{proof}
Consider the decomposition
$$ f^{\ast}D=\pp(f^{\ast}D+E)+\np(f^{\ast}D+E)-E. $$
By the general negativity lemma, $\np(f^{\ast}D+E)-E\ge 0$. Hence, by Lemma \ref{lemma:np-difference}, we can conclude that
$$ \np(f^{\ast}D)=\np(f^{\ast}D+E)-E. $$
\end{proof}

\begin{remark}[$D$-MMP] \label{remark:D-MMP}
Let $X$ be a normal projective variety and $D$ an $\R$-Cartier divisor on $X$. Let $R$ be a $D$-negative extremal ray, i.e., $R$ is an extremal ray of $\overline{\NE}(X)$ such that $D\cdot R<0$. Assume that we have a contraction $f:X\to Z$ such that for any curve $C$ on $X$, $f(C)=\pt$ if and only if $[C]\in R$. Consider the following three possibilities:
\begin{enumerate}[label=(\arabic*)]
\item (fibration) $\dim X> \dim Z$.
\item (divisorial contraction) $f$ is birational and $\Exc(f)$ is a divisor.
\item (flipping contraction) $f$ is birational and $\Exc(f)$ has codimension $\ge 2$.
\end{enumerate}
We proceed the $D$-MMP as follows: If $f:X\to Z$ is a fibration, then we stop the MMP. If $f:X\to Z$ is a divisorial contraction, then we let $X'=Z$ and $\phi=f$. If $f:X\to Z$ is a flipping contraction, then assuming the existence of the $D$-flip $f':X'\to Z$ of $f$ (cf. \cite[Definition 2.3]{Bir07}), we have a birational map $\phi:X\dashrightarrow X'$. In any case, there is a birational map $\phi:X\dashrightarrow X'$. Let $D'$ be the birational transform of $D$ on $X'$. Assume furthermore that $D'$ is $\R$-Cartier (for instance, it is true if $X'$ is $\Q$-factorial). If $D'$ is nef, then we stop. Otherwise, choose a $D'$-negative extremal ray $R'$ of $\overline{\NE}(X')$ and continue this process on $X'$ with $D'$. In this way, we obtain a sequence of birational maps
$$ X=X_{1}\dashrightarrow X_{2}\dashrightarrow \cdots $$
which is called a \textit{$D$-MMP}.
\end{remark}

The following lemma shows how positive part of divisorial Zariski decomposition behaves along the MMP.

\begin{lemma}[cf.\protect{\cite[Lemma 4.1]{BH}}] \label{lemma:positive part-ZD-MMP}
Let $X$ be a normal projective variety and $D$ a pseudoeffective $\R$-divisor on $X$. If $\phi:X\dashrightarrow X'$ is the $D$-MMP, then $\phi_{\ast}\pp(D)=\pp(\phi_{\ast}D)$.
\end{lemma}

\begin{proof}
Let $p:Y\to X$ and $q:Y\to X'$ be a common resolution. Write
$$ p^{\ast}D=q^{\ast}\phi_{\ast}D+E. $$
By assumption, $E$ is a $q$-exceptional effective divisor. By Lemma \ref{lemma:np-exceptional divisor},
$$ \np(p^{\ast}D)=\np(q^{\ast}\phi_{\ast}D+E)=\np(q^{\ast}\phi_{\ast}D)+E $$
which implies that $\pp(p^{\ast}D)=\pp(q^{\ast}\phi_{\ast}D)$.
\end{proof}

\section{Anticanonical minimal models and Zariski decompositions}
\subsection{MMP for a polarized pair}

Let $(X,B)$ be an lc pair and $P$ a nef $\R$-divisor on $X$. Let $A$ be an ample $\R$-divisor on $X$ such that $K_{X}+B+P+A$ is nef. We will construct the $(K_{X}+B+P)$-MMP with scaling of $A$ as follows (cf. \cite[Section 3]{BH}). Let $\lambda=\inf\{t\ge 0 ~|~ \text{$K_{X}+B+P+tA$ is nef}\}$. If $\lambda=0$, then there is nothing to do further. Assume that $\lambda>0$. Then for $0<s<\lambda$, we can find a boundary divisor $\Delta$ such that $(X,\Delta)$ is lc and
$$ K_{X}+\Delta\sim_{\R}K_{X}+B+P+sA. $$
There is an extremal ray $R$ of $\overline{\NE}(X)$ such that $(K_{X}+\Delta)\cdot R<0$ and $(K_{X}+\Delta+(\lambda-s)A)\cdot R=0$. By the construction, $(K_{X}+B+P)\cdot R<0$ and $(K_{X}+B+P+\lambda A)\cdot R=0$. Since $R$ is a $(K_{X}+\Delta)$-negative ray, we can contract it. If the contraction yields a fiberation, then we stop the MMP. Otherwise, we obtain either a divisorial contraction or a flip $\phi:X\dashrightarrow X'$.

Let $K_{X'}+B'+P'+\lambda A'=\phi_{\ast}(K_{X}+B+P+\lambda A)$, and define by $\lambda'=\inf\{t\ge 0 ~|~\text{$K_{X'}+B'+P'+tA'$ is nef}\}$. If $\lambda'=0$, then $K_{X'}+B'+P'$ is nef and we stop the MMP. Assume that $\lambda'>0$. Choose sufficiently small $s'$ and take a boundary divisor $\Gamma$ on $X$ such that $(X,\Gamma)$ is lc and $K_{X}+\Gamma\sim_{\R}K_{X}+B+P+s'A$. Since $s'$ is sufficiently small, we have $0<s'<\lambda'$ and $(K_{X}+\Gamma)\cdot R<0$. Let $K_{X'}+\Gamma'=\phi_{\ast}(K_{X}+\Gamma)$. Then $(X',\Gamma')$ is lc and there is an extremal ray $R'$ of $\overline{\NE}(X')$ such that $(K_{X'}+\Gamma')\cdot R'<0$ and $(K_{X'}+\Gamma'+(\lambda'-s')A')\cdot R'=0$. One can easily see that $(K_{X'}+B'+P')\cdot R'<0$ and $(K_{X'}+B'+P'+\lambda'A')\cdot R'=0$. Moreover, once again we can contract the ray $R'$ and continue the MMP.

\begin{remark} \label{remark:np=0}
Let $(X,B)$ be a klt pair and $P$ a nef $\R$-divisor on $X$ such that $K_{X}+B+P$ pseudoeffective. Let $\phi_{i}:X_{i}\dashrightarrow X_{i+1}$ be a $(K_{X}+B+P)$-MMP with scaling of an ample divisor $A$ on $X$, starting with $X_{1}=X$. Assume that the MMP terminates with a model $X'$. Then the birational transform $K_{X'}+B'+P'$ of $K_{X}+B+P$ is nef and 
$$ \np(K_{X'}+B'+P')=0. $$
Assume that there are infinitely many flips. After finitely many steps, the MMP consists of only flips, that is, there exists a positive integer $i$ such that $\phi_{j}:X_{j}\dashrightarrow X_{j+1}$ is a flip for all $j\ge i$. In this case, one can see that $\np(K_{X_{i}}+B_{i}+P_{i})=0$, where $K_{X_{i}}+B_{i}+P_{i}$ is the birational transform of $K_{X}+B+P$ (cf. \cite[Theorem 2.3]{Fuj11b}). Thus, in any case, we arrive at a model $X'$ where $\np(K_{X'}+B'+P')=0$.
\end{remark}

\subsection{$P$-trivial MMP}

Let $(X,B)$ be a $\Q$-factorial lc pair and $P$ a nef $\R$-divisor on $X$. Let $X_{i}\dashrightarrow X_{i+1}$ be a $(K_{X}+B+P)$-MMP with scaling of an ample divisor $A$, starting with $X_{1}=X$. Denote by $K_{X_{i}}+B_{i}+P_{i}$ the birational transform of $K_{X}+B+P$ on $X_{i}$ and $R_{i}$ the extremal ray of $\overline{\NE}(X_{i})$ associated to the birational map $X_{i}\dashrightarrow X_{i+1}$. We say that the MMP is $P$-trivial if $P_{i}\cdot R_{i}=0$ for all $i$. In particular, if the MMP is $P$-trivial, then $P_{i}$ is nef for all $i$. Our aim is to show that $(K_{X}+B+\alpha P)$-MMP is $P$-trivial if $\alpha$ is sufficiently large (cf. \cite{BH, Hu}).

First, we prove the simplest case. Assume that $P$ is $\Q$-Cartier. Then there is a positive integer $m$ such that $mP$ is Cartier. Let $R$ be a $(K_{X}+B+\alpha P)$-negative extremal ray of $\overline{\NE}(X)$. It is also a $(K_{X}+B)$-negative extremal ray and by the cone theorem (cf. \cite{Fuj}), we can find a rational curve $\Sigma$ such that $\Sigma$ generates $R$ and
$$ 0<-(K_{X}+B)\cdot \Sigma \le 2\dim X. $$
Suppose that $\alpha\ge 2m\dim X$. If $P\cdot \Sigma>0$, then $P\cdot \Sigma\ge\frac{1}{m}$ and we have
$$ (K_{X}+B+\alpha P)\cdot \Sigma\ge -2\dim X+\alpha P\cdot \Sigma\ge 0 $$
which is a contradiction. Thus, $P\cdot \Sigma=0$. By the cone theorem, $mP$ remains Cartier through the MMP. Consequently, we can conclude that $(K_{X}+B+\alpha P)$-MMP is $P$-trivial if $\alpha \ge 2m\dim X$.

In order to prove our main result, we need to treat more general case and it is necessary to recall the definition of extremal curve (cf. \cite{Sho2}). Let $X$ be a normal projective variety and $C$ a curve on $X$. We say that $C$ is \textit{extremal} if it generates an extremal ray $R$ of $\overline{\NE}(X)$ and has a minimal degree, that is, $C\cdot H=\min\{C'\cdot H | [C']\in R\}$ for some ample divisor $H$ on $X$. Let $B$ be a boundary divisor on $X$ such that $(X,B)$ is an lc pair and $R$ a $(K_{X}+B)$-negative extremal ray of $\overline{\NE}(X)$. By the cone theorem, we can find a rational curve $\Sigma$ such that $\Sigma$ generates the ray $R$ and
$$ 0<-(K_{X}+B)\cdot \Sigma \le 2\dim X. $$
Let $C$ be an extremal curve associated to the ray $R$ and $H$ an ample divisor on $X$. Then we have
$$ \frac{-(K_{X}+B)\cdot C}{H\cdot C}=\frac{-(K_{X}+B)\cdot \Sigma}{H\cdot \Sigma}. $$
Hence, we can conclude that
$$ 0<-(K_{X}+B)\cdot C\le 2\dim X. $$
Now, a simple observation leads to the following lemma.

\begin{lemma}
Let $(X,B)$ be an lc pair and $C$ an extremal curve on $X$. Then we have
$$ -(K_{X}+B)\cdot C\le 2\dim X. $$
\end{lemma}

\begin{lemma}[cf. {\cite[Theorem 3.2]{BH}, \cite[Theorem 3.18]{Hu}}] \label{lemma:P-trivial}
Let $(X,B)$ be a $\Q$-factorial klt pair and $P$ a nef $\R$-divisor on $X$. Assume that $N:=K_{X}+B+P$ is effective and the irrational part of $P$ is supported on $N$. Then any $(K_{X}+B+\alpha P)$-MMP is $P$-trivial for all sufficiently large $\alpha\gg 0$.
\end{lemma}

\begin{proof}
By the same argument in \cite{BH}, there is a decomposition $P=\sum_{i=1}^{k}r_{i}P_{i}$ satisfying the followings:
\begin{enumerate}[label=$\bullet$]
\item $P_{i}$ are $\Q$-Cartier,
\item $r_{i}$ are positive and linearly independent over $\Q$, and
\item $\Supp (P-P_{i})\subseteq \Supp N$.
\end{enumerate}
Note that we can chose $P_{i}$ so that $\norm{P-P_{i}}$ are arbitrarily small. Since $\Supp (P-P_{i})\subseteq \Supp N$, we may assume that $N_{i}:=K_{X}+B+P_{i}$ are effective. Let $R$ be a $(K_{X}+B+\alpha P)$-negative extremal ray and $C$ an extremal curve generating $R$. Let $\mu_{i}:=\lct(X,B,N_{i})$ be the log canonical threshold. Then we have
$$ \mu_{i}P_{i}\cdot C=\mu_{i}N_{i}\cdot C-\mu_{i}(K_{X}+B)\cdot C>\mu_{i}N_{i}\cdot C $$
and
\begin{align*}
\mu_{i}N_{i}\cdot C&=(K_{X}+B+\mu_{i}N_{i})\cdot C-(K_{X}+B)\cdot C\\
&>(K_{X}+B+\mu_{i}N_{i})\cdot C\\
&\ge -2\dim X.
\end{align*}
Let $m_{i}$ be a positive integer such that $m_{i}P_{i}$ is Cartier. Then we have
$$ P\cdot C=\sum_{i=1}^{k}\frac{r_{i}n_{i}}{m_{i}}, $$
where $n_{i}$ are integers such that $n_{i}\ge -2\mu_{i}^{-1}m_{i}\dim X$. Hence, if $\alpha$ is sufficiently large, then $P\cdot C$ should be zero. Since $r_{i}$ are linearly independent over $\Q$, if $P\cdot C=0$, then $P_{i}\cdot C=0$ for all $i$. By the cone theorem, $m_{i}P_{i}$ remains Cartier through the MMP. Note also that $N_{i}\cdot C=(K_{X}+B+P_{i})\cdot C<0$. It follows that $(K_{X}+B+\mu_{i}N_{i})\cdot R<0$ hence the MMP is a also $(K_{X}+B+\mu_{i}N_{i})$-MMP. Finally, we can conclude that the MMP is $P$-trivial for all sufficiently large $\alpha\gg 0$.
\end{proof}



\subsection{Anticanonical minimal models} \label{remark:minimal model-ZD}
The converse of Theorem \ref{theorem:main theorem} is actually easy to prove. Let $(X,\Delta)$ be a pair with $-(K_{X}+\Delta)$ a pseudoeffective $\Q$-Cartier divisor. Assume that there exists a $-(K_{X}+\Delta)$-minimal model $(X',\Delta')$ with a birational map $\phi:X\dashrightarrow X'$. Let $p:V\to X$ and $q:V\to X'$ be a common log resolution. Write
$$ -p^{\ast}(K_{X}+\Delta)=-q^{\ast}(K_{X'}+\Delta')+\Phi. $$
Then $\Phi=\sum_{E}(a(E;X,\Delta)-a(E;X',\Delta'))E$, where $E$ is a prime divisor on $V$. If $E$ is a divisor on $X$, then by definition,
$$ a(E;X,\Delta)-a(E;X',\Delta')=a(E;X,\Delta)-\bar{a}(E;X,\Delta)=\sigma_{E}(-(K_{X}+\Delta))\ge 0. $$
Therefore, $p_{\ast}\Phi\ge 0$ and by Lemma \ref{lemma:negativity lemma}, $\Phi\ge 0$. Since $\Phi$ is $q$-exceptional divisor, by Lemma \ref{lemma:np-exceptional divisor},
\begin{equation} \label{equation:*}
\np(-p^{\ast}(K_{X}+\Delta))=\np(-q^{\ast}(K_{X'}+\Delta')+\Phi)=\Phi. \tag{$\ast$}
\end{equation}
Hence, we have $\pp(-p^{\ast}(K_{X}+\Delta))=-p^{\ast}(K_{X}+\Delta)-\Phi=-q^{\ast}(K_{X'}+\Delta')$ is nef. Namely, $-(K_{X}+\Delta)$ admits a birational Zariski decomposition.

\begin{remark}
Let $(X',\Delta')$ be a $-(K_{X}+\Delta)$-minimal model. Then by the equation (\ref{equation:*}), we have
$$ \bar{a}(E;X,\Delta)=\bar{a}(E;X',\Delta') $$
for any prime divisor $E$ over $X$.
\end{remark}

Before starting the proof, we fix some notation for the convenience. Let $\phi:Y\dashrightarrow Y'$ be a birational morphism of projective varieties. For an $\R$-divisor $D_{Y}$ on $Y$, we denote by $D_{Y'}$ the birational transform of $D_{Y}$ on $Y'$ via $\phi$.

\begin{proof}[Proof of Theorem \ref{theorem:main theorem}]
Let $f:Y\to X$ be a log resolution such that $-f^{\ast}(K_{X}+\Delta)$ admits the Zariski decomposition $-f^{\ast}(K_{X}+\Delta)=P_{Y}+N_{Y}$. Let $K_{Y}+\Delta_{Y}=f^{\ast}(K_{X}+\Delta)$. Then we have
$$ K_{Y}+\Delta_{Y}+N_{Y}+P_{Y}=0 $$
and $\bar{a}(E;X,\Delta)=1-\mult_{E}(\Delta_{Y}+N_{Y})$ for a prime divisor $E$ on $Y$. By assumption, the sub-pair $(Y,\Delta_{Y}+N_{Y})$ is sub-klt. Choose a sufficiently small positive real number $\varepsilon>0$ such that $(Y,\Delta_{Y}+(1+\varepsilon)N_{Y})$ is sub-klt. 

By taking a higher model if necessary, we may assume that $\Exc(f)\cup\Supp f_{\ast}^{-1}\Delta\cup \Supp N_{Y}$ is a reduced simple normal crossing divisor.
Let $E_{1},\dots, E_{k}$ be exceptional divisors of $f$, and write $\Delta_{Y}=f_{\ast}^{-1}\Delta+\sum_{i=1}^{k}d_{i}E_{i}$. For sufficiently small real number $e>0$, define an $f$-exceptional divisor $\Gamma$ as
$$ \mult_{E_{i}}\Gamma:=\begin{cases} -d_{i} &~\text{if $d_{i}<0$}, \\ e &~\text{otherwise} \end{cases} $$
and let $B_{Y}$ be an effective divisor on $Y$ defined as
\begin{equation} \label{equation:sec4-1}
B_{Y}=\Delta_{Y}+(1+\varepsilon)N_{Y}+\Gamma.
\end{equation}
Denote by $-(K_{X}+\Delta)=P+N$ the divisorial Zariski decomposition and let $B=\Delta+(1+\varepsilon)N$. Then one can see that the followings hold:
\begin{enumerate}[label=$\bullet$]
\item $ K_{Y}+B_{Y}+(1+\varepsilon)P_{Y}=f^{\ast}(K_{X}+B+(1+\varepsilon)P)+\Gamma $,
\item $(Y,B_{Y})$ is a klt pair, and
\item $\Supp\Gamma=\Exc(f)$.
\end{enumerate}
By Lemma \ref{lemma:np-exceptional divisor}, we have
\begin{align} \label{equation:sec4-2}
\np(K_{Y}+B_{Y}+(1+\varepsilon)P_{Y})&=\np(f^{\ast}(K_{X}+B+(1+\varepsilon)P)+\Gamma) \nonumber\\
&=\np(f^{\ast}(K_{X}+B+(1+\varepsilon)P))+\Gamma \nonumber\\
&=\varepsilon N_{Y}+\Gamma.
\end{align}
Hence the positive part is given as
\begin{equation}
\pp(K_{Y}+B_{Y}+(1+\varepsilon)P_{Y})=\varepsilon P_{Y}.
\end{equation}
Let $A_{Y}$ be an ample divisor on $Y$ and $\alpha>0$ a positive real number. Run a $(K_{Y}+B_{Y}+\alpha P_{Y})$-MMP with scaling of $A_{Y}$. Since $-(K_{X}+\Delta)$ is a $\Q$-divisor, the irrational part of $P_{Y}$ is supported on $N_{Y}$. By Lemma \ref{lemma:P-trivial}, such MMP is $P_{Y}$-trivial for sufficiently large $\alpha\gg 1$. That is, the nefness of $P_{Y}$ is preserved through the MMP. By Remark \ref{remark:np=0}, we arrive at a model $Y'$ such that
\begin{equation} \label{equation:np=0}
 \np(K_{Y'}+B_{Y'}+\alpha P_{Y'})=0.
\end{equation}

Now, we will trace the positive part of $K_{Y}+B_{Y}+\alpha P_{Y}$. By Remark \ref{remark:np-addition of pp}, the positive part is given as
\begin{align*}
\pp(K_{Y}+B_{Y}+\alpha P_{Y})&=\pp(K_{Y}+B_{Y}+(1+\varepsilon)P_{Y}+(\alpha-1-\varepsilon)P_{Y})\\
&=(\alpha-1)P_{Y}.
\end{align*}
Let $\phi:Y\dashrightarrow Y'$ be the MMP. By Lemma \ref{lemma:positive part-ZD-MMP}, we obtain that
$$ \pp(K_{Y'}+B_{Y'}+\alpha P_{Y'})=\phi_{\ast}\pp(K_{Y}+B_{Y}+\alpha P_{Y})=(\alpha-1)P_{Y'}. $$
On the other hand, by (\ref{equation:np=0}), we have
\begin{align*}
\pp(K_{Y'}+B_{Y'}+\alpha P_{Y'})&=K_{Y'}+B_{Y'}+\alpha P_{Y'}\\
&=(\alpha-1)P_{Y'}+\varepsilon N_{Y'}+\Gamma_{Y'}.
\end{align*}
Therefore, $\varepsilon N_{Y'}+\Gamma_{Y'}=0$ and $K_{Y'}+B_{Y'}+P_{Y'}=0$.

Let $\psi=\phi\circ f^{-1}:X\dashrightarrow Y'$ and let $\Delta_{Y'}=\psi_{\ast}\Delta$ be the birational transform of $\Delta$ on $Y'$. We will show that $(Y',\Delta_{Y'})$ is a $-(K_{X}+\Delta)$-minimal model. Since $\Supp\Gamma=\Exc(f)$ and $\Gamma_{Y'}=0$, $\psi:X\dashrightarrow Y'$ is a birational contraction. Note that
$$ B_{Y'}=\Delta_{Y'}+(1+\varepsilon)N_{Y'}+\Gamma_{Y'}=\Delta_{Y'}. $$
Thus $-(K_{Y'}+\Delta_{Y'})=P_{Y'}$ is nef. By construction, the $(K_{Y}+B_{Y}+\alpha P_{Y})$-MMP is also $(K_{Y}+B_{Y})$-MMP. It follows that $(Y',B_{Y'}=\Delta_{Y'})$ is a $\Q$-factorial klt pair. We have shown that $(Y',\Delta_{Y'})$ satisfies the first four conditions in Definition \ref{definition:-K-minimal model}. To complete the proof, we need to show that $\bar{a}(E;X,\Delta)=\bar{a}(E;Y',\Delta_{Y'})$ for any prime divisor $E$ on $X$.

Let $p:V\to Y$ and $q:V\to Y'$ be a common log resolution. Then there is an effective $q$-exceptional diviosr $\Phi$ such that
$$ p^{\ast}(K_{Y}+B_{Y}+\alpha P_{Y})=q^{\ast}(K_{Y'}+B_{Y'}+\alpha P_{Y'})+\Phi. $$
By (\ref{equation:sec4-2}) and Lemma \ref{lemma:np-behavior}, we have
$$ \np(p^{\ast}(K_{Y}+B_{Y}+\alpha P_{Y}))=p^{\ast}(\varepsilon N_{Y}+\Gamma). $$
Since $q^{\ast}(K_{Y'}+B_{Y'}+\alpha P_{Y'})=(\alpha-1)q^{\ast}P_{Y'}$ is nef and $\Phi$ is an effective $q$-exceptional divisor, it follows from Lemma \ref{lemma:np-exceptional divisor} that
\begin{equation} \label{equation:sec4-3}
\Phi=p^{\ast}(\varepsilon N_{Y}+\Gamma)
\end{equation}
and
\begin{align*}
(\alpha-1)p^{\ast}P_{Y}&=\pp(p^{\ast}(K_{Y}+B_{Y}+\alpha P_{Y}))\\
&=\pp(q^{\ast}(K_{Y'}+B_{Y'}+\alpha P_{Y'}))\\
&=(\alpha-1)q^{\ast}P_{Y'}.
\end{align*}
Therefore $p^{\ast}P_{Y}=q^{\ast}P_{Y'}$ and $\Phi=\sum_{E}(a(E;Y',B_{Y'})-a(E;Y,B_{Y}))E$, where $E$ runs over all prime divisors on $V$. Let $E$ be a prime divisor on $Y$. Then we have
\begin{align*}
a(E;Y,B_{Y})&=1-\mult_{E}B_{Y}\\
&=a(E;X,\Delta)-(1+\varepsilon)\mult_{E}N_{Y}-\mult_E\Gamma\\
&=\bar{a}(E;X,\Delta)-\mult_{E}(\varepsilon N_{Y}+\Gamma).
\end{align*}
On the other hand, by (\ref{equation:sec4-3}), we have
\begin{align*}
a(E;Y',\Delta_{Y'})&=a(E;Y',B_{Y'})\\
&=a(E;Y,B_{Y})+\mult_{E}(\varepsilon N_{Y}+\Gamma).
\end{align*}
Hence, we can conclude that
$$ \bar{a}(E;X,\Delta)=a(E;Y,\Delta_{Y'}) $$
for any prime divisor $E$ on $Y$.
\end{proof}

\begin{remark}
Let $(X,\Delta)$ be a pair with pseudoeffective $\R$-divisor $-(K_{X}+\Delta)$. We can show that the existence of $-(K_{X}+\Delta)$-minimal model for movable $\R$-divisor $-(K_{X}+\Delta)$ implies the existence of $-(K_{X}+\Delta)$-minimal model for arbitrary pseudoeffecitve $\R$-divisor $-(K_{X}+\Delta)$.

Assume that $X$ is $\Q$-factorial, and let $-(K_{X}+\Delta)=P+N$ be the divisorial Zariski decomposition. Take a log resolution $f:Y\to X$ of $(X,\Delta+N)$. Then $E=\np(-f^{\ast}(K_{X}+\Delta))-f^{\ast}N$ is an $f$-exceptional divisor. By Lemma \ref{lemma:general negativity lemma}, we have $E\ge 0$. That is, $f^{\ast}N\le \np(-f^{\ast}(K_{X}+\Delta))$ and by Lemma \ref{lemma:np-difference}, we have
\begin{equation} \label{equation:sec5-1}
\np(f^{\ast}P)=\np(-f^{\ast}(K_{X}+\Delta)-f^{\ast}N)=\np(-f^{\ast}(K_{X}+\Delta))-f^{\ast}N=E.
\end{equation}
Write
$$ K_{Y}+\Delta_{Y}=f^{\ast}(K_{X}+\Delta). $$
Then by the construction of $E$, we have
\begin{equation} \label{equation:sec5-2}
K_{Y}+\Delta_{Y}+\np(-f^{\ast}(K_{X}+\Delta))=f^{\ast}(K_{X}+\Delta+N)+E.
\end{equation}
Note that $-(K_{X}+\Delta+N)=P$ is pseudoeffective. Furthermore, for any prime divisor $F$ on $Y$, by (\ref{equation:sec5-1}) and (\ref{equation:sec5-2}), we have
\begin{align*}
\bar{a}(F;X,\Delta+N)&=1-\mult_{F}(\Delta_{Y}+\np(-f^{\ast}(K_{X}+\Delta))-E)-\mult_{F}E\\
&=1-\mult_{F}(\Delta_{Y}+\np(-f^{\ast}(K_{X}+\Delta)))\\
&=\bar{a}(F;X,\Delta).
\end{align*}
It follows that $(X,\Delta)$ is pklt if and only if $(X,\Delta+N)$ is pklt.

Now, suppose that $(X,\Delta+N)$ is pklt. Let $\phi:X\dashrightarrow X'$ be the birational contraction such that $(X',\Delta'+N')$ is a $-(K_{X}+\Delta+N)$-minimal model and let $f:Y\to X$ and $g:Y\to X'$ be a common resolution such that $g=\phi\circ f$. We may assume that $f$ is a log resolution of $(X,\Delta+N)$. Write
$$ -f^{\ast}(K_{X}+\Delta+N)=-g^{\ast}(K_{X'}+\Delta'+N')+\Phi, $$
where $\Phi$ is a $g$-exceptional divisor. As we have seen in Remark \ref{remark:minimal model-ZD}, $\Phi$ is an effective divisor and $\np(-f^{\ast}(K_{X}+\Delta+N))=\Phi$. On the other hand, by (\ref{equation:sec5-1}),
$$ \np(-f^{\ast}(K_{X}+\Delta+N))=\np(f^{\ast}P)=E. $$
Hence $E=\Phi$. Note that
\begin{align*}
\pp(-f^{\ast}(K_{X}+\Delta))&=f^{\ast}P-E\\
&=-f^{\ast}(K_{X}+\Delta+N)-\Phi\\
&=-g^{\ast}(K_{X'}+\Delta'+N').
\end{align*}
Thus the positive part $\pp(-f^{\ast}(K_{X}+\Delta))$ is nef, that is, $-(K_{X}+\Delta)$ admits a birational Zariski decomposition. Now, by Theorem \ref{theorem:main theorem}, we can conclude that $(X,\Delta)$ has a $-(K_{X}+\Delta)$-minimal model.
\end{remark}


\end{document}